\documentclass[letterpaper,11pt,twoside,keywordsasfootnote,addressatend,noinfoline]{article}
\usepackage{fullpage}
\usepackage[english]{babel}
\usepackage{amssymb}
\usepackage{amsmath}
\usepackage{theorem}
\usepackage{epsfig}
\usepackage{subfigure}
\usepackage{mathrsfs}
\usepackage{color}
\usepackage{imsart}

\theorembodyfont{\slshape}

\newtheorem{theorem}{Theorem}
\newtheorem{example}{Example}
\newtheorem{lemma}[theorem]{Lemma}

\newenvironment{proof}{\noindent{\bf Proof.}}{\hspace*{2mm}~$\square$}
\newenvironment{proofof}[1]{\noindent{\bf Proof of #1.}}{\hspace*{2mm}~$\square$}

\newcommand{\N}{\mathbb{N}}
\newcommand{\R}{\mathbb{R}}
\newcommand{\Z}{\mathbb{Z}}
\newcommand{\A}{\mathscr{A}}
\newcommand{\C}{\mathscr{C}}
\newcommand{\E}{\mathscr{E}}

\newcommand{\G}{\mathscr{G}}
\newcommand{\V}{\mathscr{V}}
\newcommand{\n}{\hspace*{-6pt}}
\newcommand{\ep}{\epsilon}
\newcommand{\ind}{\mathbf 1}
\newcommand{\norm}[1]{|\!|#1|\!|}

\renewcommand{\d}{\mathbf{d}}
\renewcommand{\c}{\mathbf{c}}

\DeclareMathOperator{\card}{card}
\DeclareMathOperator{\uniform}{Uniform \,}

\begin{document}

\begin{frontmatter}
\title     {Probability of consensus in the multivariate Deffuant \\ model on finite connected graphs}
\runtitle  {Multivariate Deffuant model on finite connected graphs}
\author    {Nicolas Lanchier and Hsin-Lun Li}
\runauthor {Nicolas Lanchier and Hsin-Lun Li}
\address   {School of Mathematical and Statistical Sciences \\ Arizona State University, Tempe, AZ 85287, USA. \\ nicolas.lanchier@asu.edu \\ hsinlunl@asu.edu}

\maketitle

\begin{abstract} \ \
 The Deffuant model is a spatial stochastic model for the dynamics of opinions in which individuals are located on a connected graph representing
 a social network and characterized by a number in the unit interval representing their opinion.
 The system evolves according to the following averaging procedure:
 pairs of neighbors interact independently at rate one if and only if the distance between their opinions does not exceed a certain confidence threshold,
 with each interaction resulting in the neighbors' opinions getting closer to each other.
 All the mathematical results collected so far about this model assume that the individuals are located on the integers.
 In contrast, we study the more realistic case where the social network can be any finite connected graph.
 In addition, we extend the opinion space to any bounded convex subset of a normed vector space where the norm is used to measure the level of
 disagreement or distance between the opinions.
 Our main result gives a lower bound for the probability of consensus.
 Interestingly, our proof leads to a universal lower bound that depends on the confidence threshold, the opinion space~(convex subset and norm) and the
 initial distribution, but not on the size or the topology of the social network.
\end{abstract}



\end{frontmatter}


\section{Introduction}
\label{sec:intro}
 This paper is concerned with opinion dynamics on connected graphs.
 The first and most popular stochastic model in this topic is the voter model, introduced independently in~\cite{clifford_sudbury_1973, holley_liggett_1975}.
 The main mechanism in the voter model is social influence, the tendency of individuals to become more similar when they interact.
 More precisely, individuals located on the vertex set of a connected graph~(traditionally the~$d$-dimensional integer lattice) are characterized
 by one of two competing opinions, and update their opinion at rate one by simply mimicking one of their neighbors chosen uniformly at random.
 Using a duality relationship between the voter model and a system of coalescing random walks, it can be proved that the process on the infinite square
 lattice clusters in one and two dimensions whereas opinions coexist at equilibrium in higher dimensions~\cite{holley_liggett_1975}.
 While mathematicians studied analytically various aspects of the model such as the asymptotics for the cluster size in one and two
 dimensions~\cite{bramson_griffeath_1980, cox_griffeath_1986}, the spatial correlations at equilibrium in higher dimensions~\cite{bramson_griffeath_1979},
 and the occupation time of the process~\cite{cox_griffeath_1983}, social scientists and statistical physicists developed and
 studied numerically more realistic models of opinion dynamics.
 We refer to~\cite{lanchier_2017, liggett_1999} for reviews of the main results about the voter model, and to~\cite{castellano_fortunato_loreto_2009}
 for a review of more recent stochastic models of opinion dynamics introduced by applied scientists. \\
\indent Apart from social influence, an important component of opinion dynamics is homophily, the tendency to interact more frequently
 with individuals who are more similar.
 The most popular spatial model that includes social influence and homophily is probably the Axelrod model~\cite{axelrod_1997} where individuals are
 now characterized by a vector of cultural features, and interact with their neighbors at a rate proportional to the number of features they
 share~(homophily), which results in the two neighbors having one more feature in common (social influence).
 For a mathematical treatment of the Axelrod model, we refer
 to~\cite{lanchier_2012a, lanchier_moisson_2016, lanchier_scarlatos_2013, lanchier_schweinsberg_2012, li_2014}.
 Other spatial stochastic models of opinion dynamics include homophily in the form of a confidence threshold:
 individuals interact with their neighbors on the graph if and only if the level of disagreement between the two individuals before the interaction
 does not exceed a certain threshold.
 The simplest such model is the constrained voter model~\cite{vazquez_krapivsky_redner_2003}, the voter model with three opinions~(leftist,
 centrist and rightist) where leftists and rightists do not interact.
 Extensions of this model where the opinion space takes the form of a finite connected graph and the level of disagreement is measured using the
 geodesic distance on this graph were introduced and studied analytically in~\cite{lanchier_scarlatos_2017, scarlatos_2013}.
 The Deffuant model~\cite{deffuant_al_2000} and the Hegselmann-Krause model~\cite{hegselmann_krause_2002} are two other important spatial stochastic
 models that include social influence and homophily in the form of a confidence threshold. \\
\indent In the original version of the Deffuant model~\cite{deffuant_al_2000}, individuals are located on a general finite connected graph representing
 a social network and characterized by opinions that are initially chosen independently and uniformly at random in the unit interval.
 Pairs of neighbors interact at rate one if and only if the distance between their opinions before the interaction does not exceed a confidence
 threshold~$\tau$~(homophily), which results in the two neighbors' opinions getting closer to each other after the interaction~(social influence).
 Because~\cite{deffuant_al_2000} is purely based on numerical simulations, the authors only considered specific social networks:
 the complete graph and the two-dimensional torus.
 Their simulations on large graphs suggest the following conjecture for the infinite system:
 the process exhibits a phase transition at the critical threshold one-half in that a consensus is reached
 when~$\tau > 1/2$ whereas disagreements persist in the long long when~$\tau < 1/2$.
 This conjecture was first established in~\cite{lanchier_2012b} for the process on the integers using a combination of probabilistic and
 geometric techniques while a slightly stronger result was proved shortly after in~\cite{haggstrom_2012} using a different approach for part
 of the proof.
 The existence of a phase transition along with lower and upper bounds for the critical threshold were also proved for two variants of the model:
 a multivariate version where the opinion space is a finite-dimensional vector space and certain metrics are used to quantify the disagreement
 between individuals~\cite{hirscher_2014a, hirscher_2014b}, and a discrete version called the vectorial~Deffuant model also introduced
 in~\cite{deffuant_al_2000} where the opinion space is the hypercube and the disagreement between individuals is quantified using the
 Hamming distance~\cite{lanchier_scarlatos_2014}. \\
\indent In this paper, we study a version of the model where both the opinion space and the social network are fairly general.
 The opinion space is a bounded convex subset of a finite-dimensional normed vector space~(where the norm is used to measure the disagreements).
 Convexity is in fact a necessary assumption following from the model's evolution rules because future opinions must be on the segment connecting
 past opinions.
 More importantly, while~\cite{haggstrom_2012, hirscher_2014a, hirscher_2014b, lanchier_2012b, lanchier_scarlatos_2014} assume that the individuals are
 located on the integers, we follow~\cite{deffuant_al_2000} by assuming more realistically that the individuals are located on the vertex set of a general
 finite connected graph, meaning any possible real-world social networks.
 But unlike~\cite{deffuant_al_2000} that relies on numerical simulations and therefore can only look at a few specific graphs, our results apply
 to all possible finite connected graphs.
 Due to the finiteness of the graph, the existence of a phase transition at a specific critical threshold no longer holds, and we instead
 derive a general lower bound for the probability of consensus.
 Interestingly, while our bound depends on the choice of the opinion space~(convex subset and norm), it is uniform in all possible choices of the
 social network.


\section{Model description and main results}
\label{sec:results}
 The two key components of the model studied in this paper are the social network on which the individuals are located and the opinion space.
 To define these two components,
\begin{itemize}
 \item we let~$\G = (\V, \E)$ be a finite connected graph and \vspace*{2pt}
 \item we let~$\Delta \subset \R^d$ be a bounded convex subset and~$\norm{\cdot}$ be a norm on~$\R^d$.
\end{itemize}
 The multivariate Deffuant model is a continuous-time Markov chain whose state at time~$t$ is a configuration of opinions on the graph:
 $$ \xi_t : \V \to \Delta \quad \hbox{where} \quad \xi_t (x) = \hbox{opinion at vertex~$x$ at time~$t$}. $$
 Following all the previous works in this topic, we assume that the process starts from a constant product measure, meaning that the initial
 opinions~$\xi_0 (x)$, $x \in \V$, are independent and identically distributed, and we let~$X$ be the random variable with distribution
 $$ P (X \in B) = P (\xi_0 (x) \in B) \quad \hbox{for all~$x \in \V$ and all Borel subsets~$B \subset \Delta$.} $$
 The evolution rules are based on two parameters: the confidence threshold~$\tau > 0$ and the convergence parameter~$\mu \in (0, 1/2]$.
 The edges become independently active at rate one, which results in a potential update of the system at the two vertices connected
 by the active edge.
 More precisely, assuming that edge~$(x, y) \in \E$ is active at time~$t$, we let
 $$ \begin{array}{rcl}
    \xi_t (x) & \n = \n & \xi_{t-} (x) + \mu \,(\xi_{t-} (y) - \xi_{t-} (x)) \ \ind \{\norm{\xi_{t-} (x) - \xi_{t-} (y)} \leq \tau \} \vspace*{4pt} \\
    \xi_t (y) & \n = \n & \xi_{t-} (y) + \mu \,(\xi_{t-} (x) - \xi_{t-} (y)) \ \ind \{\norm{\xi_{t-} (x) - \xi_{t-} (y)} \leq \tau \} \end{array} $$
 while the opinions at the other vertices remain unchanged.
 In words, neighbors interact at rate one if and only if their opinion distance or level of disagreement before the interaction does not exceed
 the confidence threshold~$\tau$, which results in a partial averaging of their opinions by a factor~$\mu$.
 Note that the model is well-defined because the probability that different edges become active simultaneously is equal to zero. \\
\indent Our main result gives a lower bound for the probability of consensus that applies to any finite connected graph, any opinion space~(convex set
 and norm), and any initial distribution with value in the opinion space.
 To state this result, we let
 $$ \begin{array}{l} \d = \sup_{a, b \,\in \Delta} \norm{a - b} \quad \hbox{and} \quad
                     \c \in \Delta \ \ \hbox{such that} \ \ \sup_{a \in \Delta} \norm{a - \c} = \d / 2 \end{array} $$
 be the diameter and the center of the convex set~$\Delta$, respectively.
\begin{theorem}[probability of consensus] --
\label{th:consensus}
 For all~$\tau > \d / 2$,
 $$ P (\C) \geq 1 - \frac{E \,\norm{X - \c}}{\tau - \d / 2} \quad \hbox{where} \quad
       \C = \Bigg\{\lim_{t \to \infty} \,\sup_{x, y \in \V} \,\norm{\xi_t (x) - \xi_t (y)} = 0 \Bigg\}. $$
\end{theorem}
 The key to proving the theorem is to study a collection of auxiliary processes~(see~\eqref{eq:supermartingale} below) that keep track of the
 cumulative disagreement between a fixed opinion~$c \in \Delta$ and the opinions at each of the vertices at time~$t$.
 Using a triangle-type inequality~(Lemma~\ref{lem:triangle}), we first prove that all these auxiliary processes are almost surely nonincreasing,
 meaning that, for all~$c$, the averaging procedure can only decrease the overall level of disagreement between an observer with fixed opinion~$c$
 and the population~(Lemma~\ref{lem:monotone}).
 Almost sure monotonicity implies two important results:
\begin{enumerate}
 \item   The opinion model converges almost surely to a (random) limiting configuration. \vspace*{2pt}
 \item[] In addition, due to the evolution rules, each limiting configuration is characterized by a partition of the graph into connected components
         such that all the individuals in the same component share the same opinion and the distance between opinions in two adjacent components
         exceeds the confidence threshold~$\tau$~(Lemma~\ref{lem:limit}). \vspace*{4pt}
 \item   All the auxiliary processes are bounded supermartingales. \vspace*{2pt}
 \item[] In particular, one may apply the optional stopping theorem to these supermartingales and a certain stopping time~(Lemma~\ref{lem:stopping})
         to obtain a lower bound for the probability that the random partition above consists of only one set, meaning that all the individuals
         in the limiting configuration share the same opinion and consensus occurs.
\end{enumerate}
 Interestingly, our proof leads to a lower bound that depends on the confidence threshold, the opinion space~(convex set and norm) and the initial
 distribution, but not on the size or the topology of the social network.
 The probability of consensus, however, depends on the choice of the network so our lower bound is a universal bound that is uniform over
 all possible choices of the network. \\
\indent To illustrate our result, we now give two numerical examples where the lower bound in the theorem can be computed explicitly.
 In both examples, we let~$\norm{\cdot}$ be any norm on the vector space~$\R^d$ and assume that the set of opinions is the ball
 $$ \Delta = B (c, r) = \{a \in \R^d : \norm{a - c} < r \} \quad \hbox{where} \quad c \in \R^d \quad \hbox{and} \quad r > 0. $$
 In particular, the norm used to define the set of opinions is the same as the norm used to measure the distance between the opinions.
 In our first example, we assume that the opinions are initially uniformly distributed over the opinion set~($X = \uniform (\Delta)$), while in
 our second example, we assume that the initial distribution is of the form
\begin{equation}
\label{eq:triangle}
  P (\xi_0 (x) \in B) = \int_B (r - \norm{a - c}) \,d\lambda (a) \bigg/ \int_{\Delta} (r - \norm{a - c}) \,d\lambda (a)
\end{equation}
 for all vertices~$x \in \V$ and all Borel sets~$B \subset \Delta$.
 In both examples,~$c$ can be viewed as the centrist opinion.
 The initial opinions are closer to this centrist opinion in the second example than in the first example.
 Using the theorem, we get the following explicit lower bounds.
\begin{example} --
\label{ex:uniform}
 Assume~$X = \uniform (\Delta)$. Then,~$P (\C) \geq 1 - dr / (d + 1)(\tau - r)$.
\end{example}
\begin{example} --
\label{ex:triangle}
 Assume~\eqref{eq:triangle}. Then,~$P (\C) \geq 1 - dr / (d + 2)(\tau - r)$.
\end{example}
 The rest of the paper is devoted to proofs.
 In the next section, we show that the opinion model converges almost surely to a (random) limiting configuration in which neighbors either share
 the same opinion or disagree too much to interact.
 Then, we use the optional stopping theorem for supermartingales to derive the universal lower bound for the probability of consensus.
 Finally, we compute the lower bound explicitly for our two examples.
 

\section{Limiting configurations}
\label{sec:limit}
 The objective of this section is to prove that, regardless of the initial configuration, the process converges almost surely to a limiting
 configuration in which any two neighbors either share the same opinion or disagree too much to interact, i.e.,
 $$ \begin{array}{rl}
    \hbox{(P1)} & \lim_{t \to \infty} \xi_t (x) = \xi_{\infty} (x) \ \hbox{exists for all} \ x \in \V \vspace*{4pt} \\
    \hbox{(P2)} & \norm{\xi_{\infty} (x) - \xi_{\infty} (y)} \notin (0, \tau] \ \hbox{for all edges} \ (x, y) \in \E. \end{array} $$
 From now on, we let~$(X_t^c)$ be the process defined by
\begin{equation}
\label{eq:supermartingale}
  X_t^c = \sum_{x \in \V} \,\norm{\xi_t (x) - c} \quad \hbox{for all} \quad c \in \R^d.
\end{equation}
 That is, the process keeps track of the cumulative disagreement between a fixed opinion~$c$ possibly outside~$\Delta$ and the opinions at each of the vertices.
 To shorten the notation, we also let
 $$ \phi : \Delta \times \Delta \to \Delta \quad \hbox{defined as} \quad \phi (a, b) = (1 - \mu) a + \mu b = a + \mu (b - a). $$
 In particular, whenever a vertex~$x$ that has opinion~$a$ interacts with a vertex~$y$ that has a compatible opinion~$b \in B (a, \tau)$, meaning that
 the distance between the two opinions does not exceed~$\tau$, the opinion at~$x$
 becomes~$\phi (a, b)$ and the opinion at~$y$ becomes~$\phi (b, a)$.
 Although the details are somewhat more complicated, the basic idea to prove the two properties above is to show that the processes~$(X_t^c)$
 converge almost surely.
 To begin with, we prove the following lemma which is illustrated in Figure~\ref{fig:triangle} and gives two variations of
 the triangle inequality.
\begin{lemma}[triangle inequalities] --
\label{lem:triangle}
 For all~$a, b \in \Delta$ and all~$c \in \R^d$,
 $$ \begin{array}{rcl}
    \norm{\phi (a, b) - c} + \norm{\phi (b, a) - c} & \n \leq \n & \norm{a - c} + \norm{b - c} \vspace*{4pt} \\
    \norm{\phi (a, b) - c} + \norm{\phi (b, a) - c} & \n \leq \n & \norm{a - c} + \norm{b - c} - 2 \norm{\phi (a, b) - a} + \norm{a + b - 2c}. \end{array} $$
\end{lemma}
\begin{proof}
\begin{figure}[t!]
\centering
\scalebox{0.75}{\input{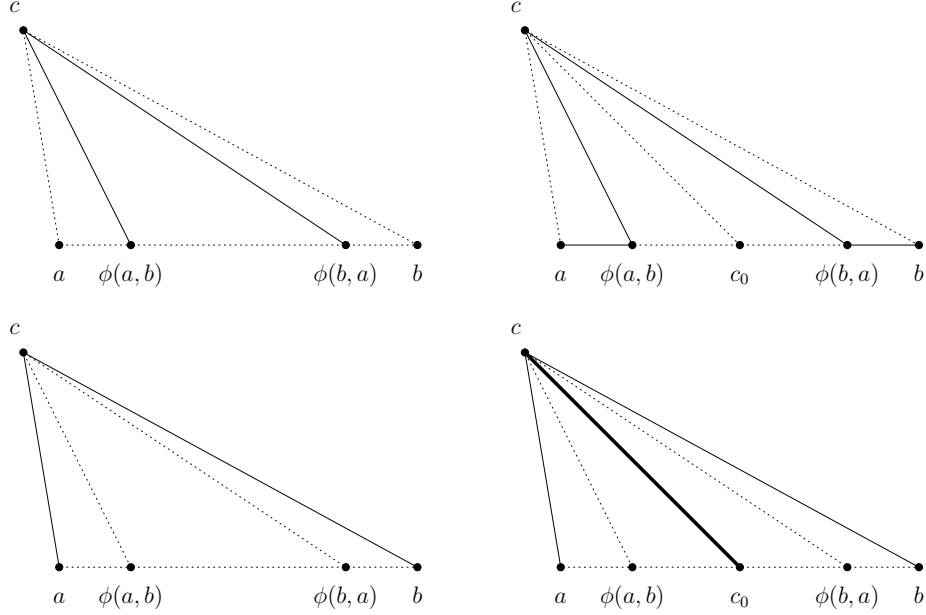}}
\caption{\upshape{Illustration of Lemma~\ref{lem:triangle}.
                  The lemma simply states that the sum of the norms of the vectors in solid lines is larger for the pictures at the bottom
                  than for the pictures at the top, where the median in thick line in the bottom right picture is counted twice.}}
\label{fig:triangle}
\end{figure}
 Using the triangle inequality and absolute homogeneity, we get
 $$ \begin{array}{l}
    \norm{\phi (a, b) - c} + \norm{\phi (b, a) - c} =
    \norm{(1 - \mu) a + \mu b - c} + \norm{(1 - \mu) b + \mu a - c} \vspace*{4pt} \\ \hspace*{80pt} =
      \norm{(1 - \mu)(a - c) + \mu (b - c)} + \norm{(1 - \mu)(b - c) + \mu (a - c)} \vspace*{4pt} \\ \hspace*{80pt} \leq
      \norm{(1 - \mu)(a - c)} + \norm{\mu (b - c)} + \norm{(1 - \mu)(b - c)} + \norm{\mu (a - c)} \vspace*{4pt} \\ \hspace*{80pt} =
      \norm{a - c} + \norm{b - c} \end{array} $$
 which proves the first inequality.
 Now, because~$0 < \mu \leq 1/2$,
 $$ a, \quad \phi (a, b), \quad c_0 = (a + b) / 2, \quad \phi (b, a), \quad b $$
 are aligned in this order, so using again the triangle inequality, we get
 $$ \begin{array}{rcl}
      \norm{\phi (a, b) - c} + \norm{\phi (b, a) - c} & \n \leq \n &
      \norm{\phi (a, b) - c_0} + \norm{c_0 - c} + \norm{\phi (b, a) - c_0} + \norm{c_0 - c} \vspace*{4pt} \\ & \n = \n &
      \norm{\phi (a, b) - \phi (b, a)} + 2 \norm{c_0 - c} \vspace*{4pt} \\ & \n = \n &
      \norm{a - b} - \norm{\phi (a, b) - a} - \norm{\phi (b, a) - b} + 2 \norm{c_0 - c} \vspace*{4pt} \\ & \n \leq \n &
      \norm{a - c} + \norm{b - c} - 2 \norm{\phi (a, b) - a} + \norm{a + b - 2c} \end{array} $$
 which proves the second inequality.
 This completes the proof.
\end{proof} \\ \\
 In the next lemma, we use the first inequality in~Lemma~\ref{lem:triangle} to prove that, for all~$c \in \Delta$, the processes~$(X_t^c)$ are
 almost surely nonincreasing.
\begin{lemma}[monotonicity] --
\label{lem:monotone}
 For all~$c \in \Delta$,
 $$ 0 \leq X_t^c \leq X_s^c \leq \d \cdot \card (\V) \quad \hbox{for all} \quad s \leq t. $$
\end{lemma}
\begin{proof}
 At each update of the processes, say at time~$s$,
 $$ \xi_s (x) = \phi (\xi_{s-} (x), \xi_{s-} (y)) \quad \hbox{and} \quad \xi_s (y) = \phi (\xi_{s-} (y), \xi_{s-} (x)) \quad \hbox{for some} \quad (x, y) \in \E. $$
 In particular, applying Lemma~\ref{lem:triangle} with~$a = \xi_{s-} (x)$ and~$b = \xi_{s-} (y)$, we get
 $$ \begin{array}{rcl}
      X_s^c - X_{s-}^c & \n = \n & \norm{\xi_s (x) - c} + \norm{\xi_s (y) - c} - \norm{\xi_{s-} (x) - c} - \norm{\xi_{s-} (y) - c} \vspace*{4pt} \\
                       & \n = \n & \norm{\phi (a, b) - c} + \norm{\phi (b, a) - c} - \norm{a - c} - \norm{b - c} \leq 0. \end{array} $$
 In addition, because~$c \in \Delta$, we have
 $$ 0 \leq X_t^c = \sum_{x \in \V} \,\norm{\xi_t (x) - c} \leq \card (\V) \cdot \sup_{a \,\in \Delta} \norm{a - c} \leq \d \cdot \card (\V) < \infty. $$
 This completes the proof.
\end{proof} \\ \\
 Note that Lemma~\ref{lem:monotone} implies that the processes~$(X_t^c)$ are bounded supermartingales, which will be used later with the optional
 stopping theorem to derive our universal lower bound for the probability of consensus.
 By the martingale convergence theorem, each of these processes converges almost surely to a finite random variable, which
 suggests almost sure convergence of the interacting particle system.
 The main difficulty to prove this result is that whenever two vertices with compatible opinions~$a$ and~$b$ interact, the process~$(X_t^c)$
 does not ``see the update'' when~$a, b, c$ are aligned in this order.
 For some norms, the lack of alignment is not even a sufficient condition for the process to see the change of opinions so it is not
 clear how to deduce convergence of the system.
 To prove this result rigorously, we now use Lemma~\ref{lem:monotone} and the second inequality in~Lemma~\ref{lem:triangle} to show that the
 jumps at each vertex get smaller and smaller as time goes to infinity.
\begin{lemma} --
\label{lem:jump}
 For all~$\ep > 0$, there exists~$S = S (\ep)$ almost surely finite such that
 $$ \norm{\xi_s (x) - \xi_{s-} (x)} < \ep \quad \hbox{for all} \quad s \geq S \ \hbox{and} \ x \in \V. $$
\end{lemma}
\begin{proof}
 Assume by contradiction that there exist~$\ep > 0$ and~$x \in \V$ such that the opinion at~$x$ jumps by more than~$\ep$ infinitely often,
 and let~$(s_i)$ be the times of these updates:
 $$ \norm{\xi_{s_i} (x) - \xi_{s_i-} (x)} \geq \ep \quad \hbox{for all} \quad i > 0. $$
 Letting~$y_i \in \V$ be the vertex that interacts with~$x$ at time~$s_i$, setting
 $$ a_i = \xi_{s_i-} (x), \quad b_i = \xi_{s_i-} (y) \quad \hbox{and} \quad c_i = (a_i + b_i) / 2, $$
 and applying the second inequality in Lemma~\ref{lem:triangle} with~$a = a_i$ and~$b = b_i$, we get
\begin{equation}
\label{eq:jump-1}
  \begin{array}{rcl}
    X_{s_i}^c - X_{s_i-}^c & \n = \n &
      \norm{\xi_{s_i} (x) - c} + \norm{\xi_{s_i} (y) - c} - \norm{\xi_{s_i-} (x) - c} - \norm{\xi_{s_i-} (y) - c} \vspace*{4pt} \\ & \n = \n &
      \norm{\phi (a_i, b_i) - c} + \norm{\phi (b_i, a_i) - c} - \norm{a_i - c} - \norm{b_i - c} \vspace*{4pt} \\ & \n \leq \n &
  - 2 \norm{\phi(a_i, b_i) - a_i} + \norm{a_i + b_i - 2c} = - 2 \norm{\xi_{s_i} (x) - \xi_{s_i-} (x)} + 2 \norm{c_i - c} \vspace*{4pt} \\ & \n \leq \n & 
  - 2 \ep + 2 \norm{c_i - c} \leq - \ep \end{array}
\end{equation}
 for all~$c \in B (c_i, \ep / 2)$.
 Now, observe that there exists~$\ep' > 0$ such that
\begin{equation}
\label{eq:jump-2}
  B (c, \ep / 2) \cap \Delta (\ep') \neq \varnothing \ \ \hbox{for all} \ \ c \in \Delta \quad \hbox{where} \quad \Delta (\ep') = \Delta \cap (\ep' \Z)^d.
\end{equation}
 For the Euclidean norm, \eqref{eq:jump-2} holds for~$\ep' = \ep / 2$.
 This and the equivalence of the norms in finite dimensions imply that, for each norm, there indeed exists~$\ep' > 0$ such
 that~\eqref{eq:jump-2} holds.
 In addition, because the opinion space~$\Delta$ is bounded, and again the dimension is finite,
\begin{equation}
\label{eq:jump-3}
  \card (\Delta (\ep')) < \infty \quad \hbox{for all} \quad \ep' > 0.
\end{equation}
 Combining~\eqref{eq:jump-2} and~\eqref{eq:jump-3}, we deduce that
 $$ \Delta' (\ep') = \{c \in \Delta (\ep') : \card \{i : c \in B (c_i, \ep / 2) \} = \infty \} \neq \varnothing. $$
 In particular, there exists
 $$ c' \in \Delta' (\ep') \quad \hbox{such that} \quad I = \{i \in \N : c \in B (c_i, \ep / 2) \} \ \hbox{is infinite}. $$
 This, together with~\eqref{eq:jump-1} and~Lemma~\ref{lem:monotone}, implies that
 $$ \lim_{t \to \infty} \,X_t^{c'} \leq X_0^{c'} + \sum_{i \in I} \,(X_{s_i}^{c'} - X_{s_i-}^{c'}) = X_0^{c'} + \sum_{i \in I} \,(- \ep) = - \infty, $$
 which contradicts the fact that~$(X_t^{c'})$ is positive.
\end{proof} \\ \\
 The next lemma shows that the jumps getting smaller and smaller implies that, for large times, neighbors must either be incompatible or
 have almost the same opinion.
\begin{lemma} --
\label{lem:distance}
 For all~$0 < \ep < \tau$, there exists~$T = T (\ep)$ almost surely finite such that
 $$ \norm{\xi_s (x) - \xi_s (y)} \notin [\ep, \tau] \quad \hbox{for all} \quad s \geq T \ \hbox{and} \ (x, y) \in \E. $$
\end{lemma}
\begin{proof}
 Assume by contradiction that there exist~$\ep > 0$ and~$(x, y) \in \E$ such that the opinion distance along the edge~$\in [\ep, \tau]$ infinitely often.
 Because neighbors interact at rate one, and therefore their opinion distance stays constant for a positive amount of time,
 $$ \int_0^{\infty} \ind \{\norm{\xi_s (x) - \xi_s (y)} \in [\ep, \tau] \} \,ds = \infty. $$
 This implies the existence of a sequence~$(s_i)$ such that~$s_i \to \infty$,
 $$ \norm{\xi_{s_i-} (x) - \xi_{s_i-} (y)} > \ep \quad \hbox{and} \quad \xi_{s_i} (x) = \phi (\xi_{s_i-} (x), \xi_{s_i-} (y)) $$
 which, in turn, implies that
 $$ \norm{\xi_{s_i} (x) - \xi_{s_i-} (x)} = \norm{\mu (\xi_{s_i-} (x) - \xi_{s_i-} (y))} > \mu \ep \quad \hbox{for all} \quad i > 0. $$
 This contradicts Lemma~\ref{lem:jump} and completes the proof.
\end{proof} \\ \\
 To deduce almost sure convergence of the particle system from the previous lemma, the last step is to prove that neighbors who almost totally
 agree cannot randomly oscillate together, which follows from an argument of convexity.
 The proof of the next lemma shows in fact a little bit more:
 there is a partition of the graph into connected components such that all the opinions in the same component are eventually trapped in a
 fixed ball with arbitrarily small radius while opinions in two adjacent components are incompatible, which implies in particular~(P1) and~(P2).
\begin{lemma} --
\label{lem:limit}
 Properties~(P1) and~(P2) hold.
\end{lemma}
\begin{proof}
 Let~$N = \card (\V)$ and~$0 < \ep < \tau / N$.
 By Lemma~\ref{lem:distance}, there exists~$T < \infty$ such that
 $$ \norm{\xi_s (x) - \xi_s (y)} \notin [\ep / N, \tau] \quad \hbox{for all} \quad s \geq T \ \hbox{and} \ (x, y) \in \E $$
 and we write~$x \leftrightarrow y$ if there exist~$x_0 = x, x_1, \ldots, x_j = y$ all distinct such that
 $$ (x_i, x_{i + 1}) \in \E \ \hbox{and} \ \norm{\xi_s (x_i) - \xi_s (x_{i + 1})} < \ep / N \quad \hbox{for all} \quad 0 \leq i < j \ \hbox{and} \ s \geq T. $$
 In particular, by the triangle inequality,
\begin{equation}
\label{eq:limit-1}
  \norm{\xi_T (x) - \xi_T (y)} \leq \sum_{i = 0}^{j - 1} \,\norm{\xi_T (x_i) - \xi_T (x_{i + 1})} < \frac{j \ep}{N} \leq \ep.
\end{equation}
 The relationship~$\leftrightarrow$ defines an equivalence relationship so it induces a partition of the vertex set
 into equivalence classes~$\V_1, \V_2, \ldots, \V_k$ that correspond to connected components of the graph.
 In addition, by~\eqref{eq:limit-1} and the definition of~$\leftrightarrow$, there exist~$c_1, c_2, \ldots, c_k \in \Delta$ such that
\begin{enumerate}
\item[(a)] for all~$i = 1, 2, \ldots, k$, we have~$\xi_T (x) \in B (c_i, \ep)$ for all~$x \in \V_i$ and \vspace*{4pt}
\item[(b)] whenever~$\V_i$ and~$\V_j$ are connected by~$(x, y) \in \E$, we have~$\norm{\xi_T (x) - \xi_T (y)} > \tau$.
\end{enumerate}
 Assume that properties~(a) and~(b) hold from time~$T$ to time~$s- > T$ and that two neighbors~$x$ and~$y$ interact at time~$s$.
 Then, either~$x \leftrightarrow y$, say~$x, y \in \V_i$, in which case
 $$ \begin{array}{rcl}
     [\xi_s (x), \xi_s (y)] & \n = \n & [(1 - \mu) \,\xi_{s-} (x) + \mu \,\xi_{s-} (y), (1 - \mu) \,\xi_{s-} (y) + \mu \,\xi_{s-} (x)] \vspace*{4pt} \\
                      & \n \subset \n & [\xi_{s-} (x), \xi_{s-} (y)] \subset B (c_i, \ep) \end{array} $$
 by convexity of~$B (c_i, \ep)$, or edge~$(x, y)$ connects two different classes in which case
 $$ \norm{\xi_{s-} (x) - \xi_{s-} (y)} > \tau \quad \hbox{therefore} \quad \xi_s (x) = \xi_{s-} (x) \ \hbox{and} \ \xi_s (y) = \xi_{s-} (y). $$
 In either case, properties~(a) and~(b) remain true after the interaction.
 Because~$\ep > 0$ can be chosen arbitrarily small, this proves that properties~(P1) and~(P2) hold.
\end{proof}

\section{Stopping time and consensus event}
\label{sec:consensus}
 This section is devoted to the proof of Theorem~\ref{th:consensus}.
 As mentioned after the proof of Lemma~\ref{lem:monotone}, the processes~$(X_t^c)$ are bounded supermartingales so the idea is to apply the
 optional stopping theorem.
 Before proving the theorem, we define a suitable stopping time and show how the consensus event relates to the configuration of the system
 at this stopping time. Let
 $$ T_* = \inf \{t : \norm{\xi_t (x) - \xi_t (y)} \notin [\tau / 2, \tau] \ \hbox{for all} \ x, y \in \V \}. $$
 Note that time~$T_*$ is a stopping time for the natural filtration of the process.
 Time~$T_*$ is also almost surely finite according to Lemma~\ref{lem:distance}, so we have the following lemma.
\begin{lemma} --
\label{lem:stopping}
 Time~$T_*$ is an almost surely finite stopping time.
\end{lemma}
 We now identify a collection of configurations at the stopping time~$T_*$ that always lead the population to consensus eventually.
 More precisely, letting
 $$ \A = \bigcup_{x \in \V} \Bigg\{\sup_{c \in \Delta} \norm{\xi_{T_*} (x) - c} < \tau \Bigg\}, $$
 we have the following inclusion.
\begin{lemma} --
\label{lem:consensus}
 We have the inclusion~$\A \subset \C$.
\end{lemma}
\begin{proof}
 The definition of~$T_*$ implies that
\begin{equation}
\label{eq:consensus-1}
 \xi_{T_*} (y) \in B (\xi_{T_*} (x), \tau) \ \Rightarrow \ \xi_{T_*} (y) \in B (\xi_{T_*} (x), \tau / 2).
\end{equation}
 In addition, by the proof of Lemma~\ref{lem:limit} (convexity argument),
\begin{equation}
\label{eq:consensus-2}
  \xi_{T_*} (y) \in B (c, \tau / 2) \ \hbox{for all} \ y \in \V \ \Rightarrow \ \xi_{T_*} (y) \in B (c, \tau / 2) \ \hbox{for all} \ y \in \V \ \hbox{and} \ s > T_*.
\end{equation}
 This, together with Lemma~\ref{lem:limit} itself, gives the implications
 $$ \begin{array}{l}
    \sup_{c \in \Delta} \norm{\xi_{T_*} (x) - c} < \tau \ \hbox{for some} \ x \in \V \vspace*{4pt} \\
    \hspace*{18pt} \Rightarrow \ (\xi_{T_*} (y) \in B (\xi_{T_*} (x), \tau) \ \hbox{for all} \ y \in \V) \ \hbox{for some} \ x \in \V \vspace*{4pt} \\
    \hspace*{18pt} \Rightarrow \ (\xi_{T_*} (y) \in B (\xi_{T_*} (x), \tau / 2) \ \hbox{for all} \ y \in \V) \ \hbox{for some} \ x \in \V \quad \hbox{(by \eqref{eq:consensus-1})} \vspace*{4pt} \\
    \hspace*{18pt} \Rightarrow \ (\xi_{T_*} (y) \in B (c, \tau / 2) \ \hbox{for all} \ y \in \V) \ \hbox{for some} \ c \in \Delta \vspace*{4pt} \\
    \hspace*{18pt} \Rightarrow \ (\xi_s (y) \in B (c, \tau / 2) \ \hbox{for all} \ y \in \V \ \hbox{and} \ s > T_*) \ \hbox{for some} \ c \in \Delta \quad \hbox{(by \eqref{eq:consensus-2})} \vspace*{4pt} \\
    \hspace*{18pt} \Rightarrow \ (\lim_{s \to \infty} \xi_s (y) \in B (c, \tau / 2) \ \hbox{for all} \ y \in \V) \ \hbox{for some} \ c \in \Delta \vspace*{4pt} \\
    \hspace*{18pt} \Rightarrow \ \lim_{s \to \infty} \norm{\xi_s (y) - \xi_s (z)} = 0 \ \hbox{for all} \ y, z \in \V \quad \hbox{(by (P2) and choice of $\tau / 2$)}. \end{array} $$
 This completes the proof.
\end{proof} \\ \\
 Using Lemmas~\ref{lem:monotone}, \ref{lem:stopping} and~\ref{lem:consensus}, we can now deduce the theorem. \\ \\
\begin{proofof}{Theorem~\ref{th:consensus}}
 According to Lemma~\ref{lem:monotone}, for all~$c \in \Delta$, the processes~$(X_t^c)$ is bounded and almost surely nonincreasing.
 In particular, the process is a bounded supermartingale with respect to the natural filtration of the opinion model.
 According to Lemma~\ref{lem:stopping}, we also have that the random time~$T_*$ is an almost surely finite stopping time with respect to the same filtration.
 In particular, it follows from the optional stopping theorem that, for all~$c \in \Delta$,
\begin{equation}
\label{eq:theorem-1}
  E (X_{T_*}^c) \leq E (X_0^c) = E \bigg(\sum_{x \in \V} \,\norm{\xi_0 (x) - c} \bigg) = \card (\V) \cdot E \,\norm{X - c}.
\end{equation}
 Now, on the complement of~$\A$,
 $$ \hbox{for all~$x \in \V$}, \quad \hbox{there exists~$c_x \in \Delta$} \quad \hbox{such that} \quad \norm{\xi_{T_*} (x) - c_x} \geq \tau. $$
 This and the triangle inequality imply that
 $$ \norm{\xi_{T_*} (x) - \c} \geq \norm{\xi_{T_*} (x) - c_x} - \norm{c_x - \c} \geq \tau - \d / 2 \quad \hbox{for all} \quad x \in \Delta. $$
 This gives the following bound for the conditional expectation:
\begin{equation}
\label{eq:theorem-2}
  E (X_{T_*}^{\c} \,| \,\A^c) = E \bigg(\sum_{x \in \V} \,\norm{\xi_{T_*} (x) - \c} \,\Big| \,\A^c \bigg) \geq \card (\V) \cdot (\tau - \d / 2).
\end{equation}
 Combining~\eqref{eq:theorem-1} with~$c = \c$ and~\eqref{eq:theorem-2}, we deduce that
 $$ \bigg(\tau - \frac{\d}{2} \bigg)(1 - P (\A)) \leq \frac{E (X_{T_*}^{\c} \,| \,\A^c) \,P (\A^c)}{\card (\V)} \leq \frac{E (X_{T_*}^{\c})}{\card (\V)} \leq E \,\norm{X - \c} $$
 which, together with Lemma~\ref{lem:consensus}, implies that
 $$ P (\C) \geq P (\A) \geq 1 - \frac{E \,\norm{X - \c}}{\tau - \d / 2} \quad \hbox{for all} \quad \tau > \d / 2. $$
 This completes the proof of the theorem.
\end{proofof}


\section{Numerical examples}
\label{sec:examples}
 In this section, we use Theorem~\ref{th:consensus} to prove Examples~\ref{ex:uniform} and~\ref{ex:triangle}. \\ \\
\begin{proofof}{Example~\ref{ex:uniform}}
 To deal with the uniform case, we first observe that
\begin{equation}
\label{eq:volume}
  \lambda (B (c, s)) = \lambda (B (0, s)) = s^d \,\lambda (B (0, 1)).
\end{equation}
 In particular, letting~$X = \uniform (\Delta)$, we get
 $$ P (\norm{X - c} < s) = P (X \in B (c, s)) = \frac{\lambda (B (c, s))}{\lambda (B (c, r))} = \bigg(\frac{s}{r} \bigg)^d $$
 for all~$s \leq r$, from which it follows that
 $$ E \,\norm{X - c} = \int_0^{\infty} P (\norm{X - c} > s) \,ds = \int_0^r \bigg(1 - \bigg(\frac{s}{r} \bigg)^d \bigg) \,ds = \frac{dr}{d + 1}. $$
 Observing also that~$\d = 2r$ and~$\c = c$, and applying Theorem~\ref{th:consensus}, we get
 $$ P (\C) \geq 1 - \frac{E \,\norm{X - \c}}{\tau - \d / 2} = 1 - \frac{E \,\norm{X - c}}{\tau - r} = 1 - \frac{dr}{(d + 1)(\tau - r)}. $$
 This completes the proof.
\end{proofof} \\ \\
\begin{proofof}{Example~\ref{ex:triangle}}
 To begin with, we observe that
 $$ \int_{\Delta} (r - \norm{a - c}) \,d\lambda (a) = \frac{r \,\lambda (B (c, r))}{d + 1}. $$
 Using also~\eqref{eq:volume} and thinking of the probability that~$X \in B (c, s)$ as the volume of a cone plus the volume of a cylinder in
 dimension~$d + 1$, we deduce that
 $$ \begin{array}{rcl}
      P (\norm{X - c} < s) & \n = \n & P (X \in B (c, s)) =
    \displaystyle \frac{s \,\lambda (B (c, s)) / (d + 1) + (r - s) \,\lambda (B (c, s))}{r \,\lambda (B (c, r)) / (d + 1)} \vspace*{8pt} \\ & \n = \n &
    \displaystyle \bigg(\frac{s}{r} \bigg)^d \bigg(\frac{s + (d + 1)(r - s)}{r} \bigg) =
    \displaystyle \bigg(\frac{s}{r} \bigg)^d \bigg(1 + d \bigg(1 - \frac{s}{r} \bigg) \bigg) \end{array} $$
 for all~$s \leq r$, from which it follows that
 $$ \begin{array}{rcl}
      E \,\norm{X - c} & \n = \n & \displaystyle \int_0^{\infty} P (\norm{X - c} > s) \,ds \vspace*{4pt} \\
                       & \n = \n & \displaystyle \int_0^r \bigg(1 - \bigg(\frac{s}{r} \bigg)^d \bigg(1 + d \bigg(1 - \frac{s}{r} \bigg) \bigg) \bigg) \,ds =
                                   \frac{dr}{d + 2}. \end{array} $$
 As previously,~$\d = 2r$ and~$\c = c$ so, according to Theorem~\ref{th:consensus}, we have
 $$ P (\C) \geq 1 - \frac{E \,\norm{X - \c}}{\tau - \d / 2} = 1 - \frac{E \,\norm{X - c}}{\tau - r} = 1 - \frac{dr}{(d + 2)(\tau - r)}. $$
 This completes the proof.
\end{proofof}


\bibliographystyle{plain}
\bibliography{biblio.bib}

\begin{thebibliography}{10}

\bibitem{axelrod_1997}
R.~Axelrod.
\newblock The dissemination of culture: A model with local convergence and
  global polarization.
\newblock {\em Journal of conflict resolution}, 41(2):203--226, 1997.

\bibitem{bramson_griffeath_1979}
M.~Bramson and D.~Griffeath.
\newblock Renormalizing the {$3$}-dimensional voter model.
\newblock {\em Ann. Probab.}, 7(3):418--432, 1979.

\bibitem{bramson_griffeath_1980}
M.~Bramson and D.~Griffeath.
\newblock Clustering and dispersion rates for some interacting particle systems
  on {${\bf Z}$}.
\newblock {\em Ann. Probab.}, 8(2):183--213, 1980.

\bibitem{castellano_fortunato_loreto_2009}
C.~Castellano, S.~Fortunato, and V.~Loreto.
\newblock Statistical physics of social dynamics.
\newblock {\em Reviews of modern physics}, 81(2):591, 2009.

\bibitem{clifford_sudbury_1973}
P.~Clifford and A.~Sudbury.
\newblock A model for spatial conflict.
\newblock {\em Biometrika}, 60:581--588, 1973.

\bibitem{cox_griffeath_1983}
J.~Theodore Cox and David Griffeath.
\newblock Occupation time limit theorems for the voter model.
\newblock {\em Ann. Probab.}, 11(4):876--893, 1983.

\bibitem{cox_griffeath_1986}
J.~Theodore Cox and David Griffeath.
\newblock Diffusive clustering in the two-dimensional voter model.
\newblock {\em Ann. Probab.}, 14(2):347--370, 1986.

\bibitem{deffuant_al_2000}
G.~Deffuant, D.~Neau, F.~Amblard, and G.~Weisbuch.
\newblock Mixing beliefs among interacting agents.
\newblock {\em Advances in Complex Systems}, 3(01n04):87--98, 2000.

\bibitem{haggstrom_2012}
O.~H\"{a}ggstr\"{o}m.
\newblock A pairwise averaging procedure with application to consensus
  formation in the {D}effuant model.
\newblock {\em Acta Appl. Math.}, 119:185--201, 2012.

\bibitem{hegselmann_krause_2002}
R.~Hegselmann and U.~Krause.
\newblock Opinion dynamics and bounded confidence models, analysis, and
  simulation.
\newblock {\em Journal of artificial societies and social simulation}, 5(3),
  2002.

\bibitem{hirscher_2014a}
T.~Hirscher.
\newblock {\em Consensus formation in the Deffuant model}.
\newblock 2014.
\newblock Thesis (Ph.D.)--Chalmers University of Technology.

\bibitem{hirscher_2014b}
T.~Hirscher.
\newblock The {D}effuant model on {$\Bbb Z$} with higher-dimensional opinion
  spaces.
\newblock {\em ALEA Lat. Am. J. Probab. Math. Stat.}, 11(1):409--444, 2014.

\bibitem{holley_liggett_1975}
R.~A. Holley and T.~M. Liggett.
\newblock Ergodic theorems for weakly interacting infinite systems and the
  voter model.
\newblock {\em Ann. Probability}, 3(4):643--663, 1975.

\bibitem{lanchier_2012a}
N.~Lanchier.
\newblock The {A}xelrod model for the dissemination of culture revisited.
\newblock {\em Ann. Appl. Probab.}, 22(2):860--880, 2012.

\bibitem{lanchier_2012b}
N.~Lanchier.
\newblock The critical value of the {D}effuant model equals one half.
\newblock {\em ALEA Lat. Am. J. Probab. Math. Stat.}, 9(2):383--402, 2012.

\bibitem{lanchier_2017}
N.~Lanchier.
\newblock {\em Stochastic modeling}.
\newblock Universitext. Springer, Cham, 2017.

\bibitem{lanchier_moisson_2016}
N.~Lanchier and P.-H. Moisson.
\newblock Fixation results for the two-feature {A}xelrod model with a variable
  number of opinions.
\newblock {\em J. Theoret. Probab.}, 29(4):1554--1580, 2016.

\bibitem{lanchier_scarlatos_2013}
N.~Lanchier and S.~Scarlatos.
\newblock Fixation in the one-dimensional {A}xelrod model.
\newblock {\em Ann. Appl. Probab.}, 23(6):2538--2559, 2013.

\bibitem{lanchier_scarlatos_2014}
N.~Lanchier and S.~Scarlatos.
\newblock Clustering and coexistence in the one-dimensional vectorial
  {D}effuant model.
\newblock {\em ALEA Lat. Am. J. Probab. Math. Stat.}, 11(1):541--564, 2014.

\bibitem{lanchier_scarlatos_2017}
N.~Lanchier and S.~Scarlatos.
\newblock Limiting behavior for a general class of voter models with confidence
  threshold.
\newblock {\em ALEA Lat. Am. J. Probab. Math. Stat.}, 14(1):63--92, 2017.

\bibitem{lanchier_schweinsberg_2012}
N.~Lanchier and J.~Schweinsberg.
\newblock Consensus in the two-state {A}xelrod model.
\newblock {\em Stochastic Process. Appl.}, 122(11):3701--3717, 2012.

\bibitem{li_2014}
J.~Li.
\newblock {\em Axelrod's model in two dimensions}.
\newblock ProQuest LLC, Ann Arbor, MI, 2014.
\newblock Thesis (Ph.D.)--Duke University.

\bibitem{liggett_1999}
T.~M. Liggett.
\newblock {\em Stochastic interacting systems: contact, voter and exclusion
  processes}, volume 324 of {\em Grundlehren der Mathematischen
  Wissenschaften}.
\newblock Springer-Verlag, Berlin, 1999.

\bibitem{scarlatos_2013}
S.~Scarlatos.
\newblock {\em Voter models with confidence parameter}.
\newblock 2013.
\newblock Thesis (Ph.D.)--University of Patras, Greece.

\bibitem{vazquez_krapivsky_redner_2003}
F.~Vazquez, P.~L. Krapivsky, and S.~Redner.
\newblock Constrained opinion dynamics: freezing and slow evolution.
\newblock {\em J. Phys. A}, 36(3):L61--L68, 2003.

\end{thebibliography}

\end{document}